\documentclass[12pt,reqno]{amsart}
\usepackage{amsmath, amsthm, amsopn, amssymb, microtype}
\usepackage{mathrsfs} 
\usepackage{mathscinet}
\usepackage{bbm}
\usepackage{enumerate, tikz, etoolbox, intcalc, geometry, caption, subcaption, afterpage}
\usepackage[english]{babel}
\usepackage{enumitem}

\usepackage[colorlinks=true,urlcolor=blue,pdfborder={0 0 0}]{hyperref}
\hypersetup{linkcolor=[rgb]{0,0,0.6}}
\hypersetup{citecolor=[rgb]{0,0.6,0}}
\usepackage[nameinlink]{cleveref}
\crefname{theorem}{Theorem}{Theorems}
\crefname{thm}{Theorem}{Theorems}
\crefname{mainthm}{Theorem}{Theorems}
\crefname{conj}{Conjecture}{Theorems}
\crefname{lemma}{Lemma}{Lemmas}
\crefname{lem}{Lemma}{Lemmas}
\crefname{remark}{Remark}{Remarks}
\crefname{prop}{Proposition}{Propositions}
\crefname{defn}{Definition}{Definitions}
\crefname{corollary}{Corollary}{Corollaries}
\crefname{cor}{Corollary}{Corollaries}
\crefname{section}{Section}{Sections}
\crefname{figure}{Figure}{Figures}
\crefname{quest}{Question}{Questions}

\newcommand{\N}{\mathbb{N}}
\newcommand{\Z}{\mathbb{Z}}

\newcommand{\Tile}{\mathit{Tile}}

\graphicspath{ {./Pictures/} }

\newcommand{\ignore}[1]{}

\newtheorem{thm}{Theorem}[section]
\newtheorem{lemma}[thm]{Lemma}

\newtheorem{cor}[thm]{Corollary}

\theoremstyle{definition}

\newtheorem{remark}[thm]{Remark}

\newif\ifdraft\drafttrue
\begin{document}

\title{A note on reduction of tiling problems} 
\author{Tom Meyerovitch}
\address{Ben-Gurion University of the Negev,
	Department of Mathematics,
	Beer-Sheva, 8410501, Israel.
	{\tt mtom@bgu.ac.il}
}

\author{Shrey Sanadhya} 
\address{Ben-Gurion University of the Negev,
	Department of Mathematics,
	Beer-Sheva, 8410501, Israel.
	{\tt sanadhya@post.bgu.ac.il}
}

\author{Yaar Solomon}
\address{Ben-Gurion University of the Negev.
	Department of Mathematics.
	Beer-Sheva, 8410501, Israel. 
	{\tt yaars@bgu.ac.il}
}


\begin{abstract}
We show that translational tiling problems in a  quotient of $\Z^d$ can be effectively reduced or ``simulated'' by translational tiling problems in $\Z^d$. In particular, for any $d \in \mathbb{N}$, $k < d$ and $N_1,\ldots,N_k \in \mathbb{N}$  the existence of an aperiodic tile in $\Z^{d-k} \times (\Z / N_1\Z \times \ldots \times \Z / N_k \Z)$ implies the existence of an aperiodic tile in $\Z^d$. Greenfeld and Tao have recently disproved the well-known periodic tiling conjecture in $\Z^d$ for sufficiently large $d \in \mathbb{N}$ by constructing an aperiodic tile in $\Z^{d-k} \times (\Z / N_1\Z \times \ldots \times \Z / N_k \Z)$ for suitable $d,N_1,\ldots,N_k \in \mathbb{N}$.
\end{abstract}

\maketitle

\section{Introduction} 

In this note, we consider translational tiles over finitely generated abelian groups. 
Any finitely generated abelian group is a quotient of $\Z^d$ for some positive integer $d$. 
Briefly, we show that for any tuple of finite subsets $F_1,\ldots,F_k$ in a quotient $\Gamma$ of $\Z^d$, there exists an explicit tuple of finite subsets $\tilde F_1,\ldots,\tilde F_k$ of $\Z^d$ so that tilings of $\Z^d$ by $\tilde F_1,\ldots,\tilde F_k$ directly correspond, formally and explicitly, to tilings of $\Gamma$ by $F_1,\ldots,F_k$. Informally, this means that tiling problems in $\Z^d$ can effectively simulate tiling problems in $\Gamma$. 

The motivation for writing this note comes from the periodic tiling conjecture and questions about the decidability of the tiling problems. Bhattacharya \cite{BPeriodicity2020} proved that a finite subset $F \subset \Z^2$ can tile $\Z^2$ if and only if it admits a periodic tiling, resolving the $\Z^2$ case of the well-known periodic tiling conjecture. 
The periodic tiling conjecture seems to be currently open in $\Z^3$ and also in $\Z^2 \times \Z / N\Z$ for general $N \in \mathbb{N}$; see \cite[subsection $ 1.3$]{Greenfeld_Tao2_2021}. From our argument below, it follows that an affirmative solution of the periodic tiling conjecture in  $\Z^3$ would automatically imply an affirmative resolution of the periodic tiling conjecture in $\Z^2 \times \Z / N\Z$ for all $N \in \mathbb{N}$.

The main result of  \cite{Greenfeld_Tao2_2021} is the existence of a rank $2$ finitely generated group $\Gamma$ and two finite subsets $F_1,F_2 \subset \Gamma$ that can tile a certain periodic subset $E \subseteq\Gamma$, but such that no proof of the sentence ``there is a tiling of $E$ by $F_1$ and $F_2$'' exists in ZFC. In particular, such sets $F_1,F_2$ cannot tile $E$ periodically. Building on this construction, it was further shown in \cite{Greenfeld_Tao2_2021} that a similar statement holds when replacing $\Gamma$ by $\Z^d$ for sufficiently large $d$. The reduction procedure described below is a slight elaboration of the construction applied by   Greenfeld and Tao in \cite{Greenfeld_Tao2_2021}. The main new point is that the reduction procedure can be used in a general setting, independently of any properties of the given set of tiles, the number of tiles, the rank $d$ of $\Z^d$, and the quotient group $\Gamma$ of $\Z^d$.

Greenfeld and Tao \cite{Greenfeld_Tao_2022} have recently disproved the periodic tiling conjecture in $\Z^d$ for sufficiently large $d \in \mathbb{N}$ by constructing an aperiodic tile in  $\Z^{d-k} \times (\Z / N_1\Z \times \ldots \times \Z / N_k \Z)$, for suitable $d,k,N_1,\ldots,N_k \in \mathbb{N}$.

We briefly describe our notational conventions that closely follow the conventions of \cite{Greenfeld_Tao2_2021}.
We write $\biguplus_{i \in I}A_i$ to denote the union of pairwise disjoint sets. 
Given an abelian group $\Gamma$ and subsets $A,B \subseteq \Gamma$, we write $A \oplus B= C$ to indicate that for every $c \in C$ there exists a unique pair $a \in A$ and $b \in B$ such that $c=a+b$. For $N \in \mathbb{N}$ and $A \subseteq \Z^d$, we denote
$N A := \{ N v :~ v \in A\}$.
%For a countable abelian group $\Gamma$ the notation $F\Subset \Gamma$ indicates that $F$ is a finite subset of $\Gamma$. 
Given finite subsets $F_1,\ldots,F_k \subset \Gamma$ and $E \subseteq \Gamma$ we use the notation 
\begin{equation}
\label{eq:GT_notation_space_of_co-tiles}
\begin{split}
&\Tile(F_1,\ldots,F_k; E)= \left\{
(A_1,\ldots,A_k) \in (2^\Gamma)^k :~ \biguplus_{j=1}^k (F_j \oplus A_j) = E
\right\}, \\
&\text{ and } \\
&\Tile_0(F_1,\ldots,F_k; E)= \left\{
(A_1,\ldots,A_k) \in \Tile(F_1,\ldots,F_k; E) :~ 0 \in \biguplus_{j=1}^k  A_j \right\}.
\end{split}
\end{equation}

We say that a $k$-tuple $(A_1,\ldots,A_k) \in (2^\Gamma)^k$ is \emph{periodic} if  there exist a finite index subgroup $\Gamma_0 < \Gamma$ that fixes every $A_j$. Equivalently, $(A_1,\ldots,A_k) \in (2^\Gamma)^k$ is periodic if there exists finite sets $C_1,\ldots,C_k \subset \Gamma$ such that $A_j = C_j \oplus \Gamma_0$ for all $1\le j \le k$. A $k$-tuple  $(F_1,\dots,F_k)$ of finite subsets of $\Gamma$ is \emph{aperiodic} if 
$\Tile(F_1,\dots,F_k; \Gamma)$ is non-empty but does not contain a periodic $k$-tuple. Given a subgroup $L < \Gamma$ we say that $D \subseteq \Gamma$ is a \emph{fundamental domain} for $L$ if $D$ contains exactly one representative of each coset of $L$. Equivalently,   $D \subseteq \Gamma$ is a fundamental domain if $D \oplus L = \Gamma$.

Here is the statement of the main reduction result:

\begin{thm}\label{thm:tiling_reduction}
Let $d,k \in \N$, $\pi:\Z^d \to \Gamma$ a surjective group homomorphism and let $D \subseteq \Z^d$ be a fundamental domain for $\ker(\pi)$.
Then there exists $N \in \N$ and finite subsets $T_1,\ldots,T_{2k} \subset \Z^d$ 
 such that for every $F_1,\ldots,F_k$, finite subsets of $\Gamma$ with $0 \in F_j$ for all $1\le j \le k$, we have
\begin{equation}\label{eq:tiling_reduction1}
\begin{split}
& \Tile_0(\tilde F_1,\ldots,\tilde F_k; \Z^d) = 
N \Big( (\pi^{\otimes k})^{-1} \Tile_0\left(F_1,\ldots,F_k; \Gamma\right) \Big), \\
 & %\qquad\qquad\qquad 
 \text{where for } 1\le j\le k \\
& %\qquad\qquad 
\tilde F_j := \left(N (\pi^{-1} (F_j\setminus \{0\}) \cap D) \oplus T_j\right) \uplus  T_{k+j} \quad \subset \Z^d.
\end{split}
\end{equation}
Here $\pi^{\otimes k}:(\Z^d)^k \to \Gamma^k$ is given by \[\pi^{\otimes k}(v_1,\ldots,v_k)=(\pi(v_1),\ldots,\pi(v_k)) \mbox{ for } v_1,\ldots,v_k \in \Z^d.\]
\end{thm}

We state some direct corollaries of \Cref{thm:tiling_reduction}. Suppose $\pi:\Z^d \to \Gamma$ is a surjective group homomorphism.

\begin{cor}\label{cor:periodicity_reduction}
If there exists an aperiodic $k$-tuple for $\Gamma$, then there exists an aperiodic $k$-tuple for $\Z^d$.
\end{cor}

\begin{cor}\label{cor:algorithmic_decidability_reduction}
For every $k\in\N$, if tiling by $k$ tiles in $\Z^d$ is algorithmically decidable, then tiling by $k$ tiles in $\Gamma$  is algorithmically decidable. 
\end{cor}

\begin{cor}\label{cor:logical_decidability}
Let $F_1,\ldots,F_k$ be finite subsets of $\Gamma$  and let $\mathcal{T}$ be a  theory in first-order logic in which the proof of \eqref{eq:tiling_reduction1} can be expressed. If the sentence   $\Tile(\tilde F_1,\ldots,\tilde F_k;\Z^d) \ne \emptyset$ is provable in $\mathcal{T}$, then the sentence 
$\Tile(F_1,\ldots,F_k;\Gamma) \ne \emptyset$ is provable in $\mathcal{T}$.
\end{cor}
Readers familiar with logic can easily convince themselves that 
the proof of \eqref{eq:tiling_reduction1} can be expressed in ZFC, for any explicit finitely generated abelian group $\Gamma$ (in the sense of \cite{Greenfeld_Tao2_2021}) and any explicit homomorphism $\pi:\Z^d \to \Gamma$.% and any explicit finite subsets $F_1,\ldots,F_k \in \Gamma$. 

{\bf Acknowledgment:} We thank Rachel Greenfeld and Terry Tao for their encouragement and Ville Salo for helpful suggestions regarding the presentation. This research was partially supported by the Israel Science Foundation grant no. 1058/18.

\section{Proof of the reduction theorem}\label{sec:proofs}
The proof of \Cref{thm:tiling_reduction} is based on the following lemma. The case $k=1$  and $L= \{0\}$ is essentially Lemma $9.3$ of \cite{Greenfeld_Tao2_2021}), where the basic idea has been attributed to Golomb \cite{MR255431}.

\begin{lemma}[Rigid tuple of tiles]\label{lem:rigid_tiles}
Let $d,s \in \N$, let $L < \Z^d$,
and let $D \subseteq \Z^d$ be a fundamental domain for $L$. Then
there exists $N \in \N$ and finite sets $T,T_1,\ldots,T_s \subset \Z^d$, with $0 \in T$ and $0 \in T_j$ for every $1\le j \le s$, such that
\begin{enumerate} [label = (\alph*)]
    \item 
    $\Tile_0(T;\Z^d) = \{N\Z^d\}$.
    \item
    For every $1 \le j \le s$ we have $T_j \oplus NL = T \oplus NL$.
    \item
    $(\tilde C_1,\dots,\tilde C_s) \in \Tile_0(T_1,\ldots,T_s; \Z^d)$ if and only if $\biguplus_{j=1}^s \tilde C_j= N\Z^d$ and each $\tilde C_j$ is $NL$-periodic.
\end{enumerate} 
\end{lemma}

\begin{remark}
\begin{enumerate}[label = (\roman*)]
    \item 
    The tiles $T,T_1,\ldots,T_s$ that are defined in the proof below are essentially boxes of side length $N$, with bumps and dents, such as pieces of a jigsaw puzzle, whose purpose is to enforce their positions (see \Cref{Fig:T}). For simplicity, the bumps and dents in the proof are of the form of ``a frame" of width $1$, i.e. the difference between two centered boxes, $B_n\setminus B_{n-1}$. We did not attempt to minimize the number $N$.  
    \item
    We note that as a consequence of (c), if $(\tilde C_1,\dots,\tilde C_s) \in \Tile_0(T_1,\ldots,T_s; \Z^d)$ then any permutation of $\{\tilde C_1,\dots,\tilde C_s\}$ is also in $\Tile_0(T_1,\ldots,T_s; \Z^d)$.
    \end{enumerate}
\end{remark}

\begin{proof}
We prove the claim only for $d \ge 2$, since the proof for the case $d=1$ requires slightly different considerations but is not more difficult.

Let $d,s \in \N$ and $L < \Z^d$ be given. Since $L$ is a finitely generated abelian group of rank at most $d$, there exists $r \le d$ and $w_1,\ldots, w_r \in L$  that generate $L$ as a free abelian group in the  sense that 
$L = \Z w_1 \oplus \ldots \oplus \Z w_r$. Also, let $e_1,\ldots,e_d \in \Z^d$ be a set of free generators for $\Z^d$ (say, the standard generators), and for $n\ge 0$, let 
\[
B_n:= \{-n,\ldots,n\}^d \subset \Z^d
\]
and
\[ 
S_n:= B_n \setminus B_{n-1} \mbox{ for } n \ge 1.
\] 

The $S_i$'s play the role of the bumps and dents in our tiles, and the properties that we need from them are that they are all bounded and that for every $i\neq j$ the set $S_i$ does not contain a translated copy of the set $S_j$. 

Choose $m \in \N$ large enough so that $B_m$ contains disjoint translated copies of all the boxes $B_l$ for $1\le l\le d+rs$ (e.g. $m \ge (d+rs)^2$), and let $N= 2m+1$.

Choose translation vectors $v_1,\ldots,v_{d+rs}\in B_m$ so that 
\begin{align*}
& \forall 1 \le i \le d+rs: \quad  (v_i + B_{d+rs}) \subseteq B_m, \\
& \text{and }  \\
& \forall 1 \le i_1 < i_2 \le d+rs: \quad (v_{i_1} + B_{d+rs}) \cap (v_{i_2} + B_{d+rs}) = \emptyset. 
\end{align*}

We first define the tile $T \subset \Z^d$ as follows: 
\begin{equation}\label{eq:T}
T: =  \left( B_m \setminus \biguplus_{i=1}^{d}\left(v_i + S_{i} \right) \right) \uplus \biguplus_{i=1}^d\left(N e_i + v_i + S_{i} \right).
\end{equation}

\begin{figure}%[ht!] 
        \includegraphics[scale=0.4]{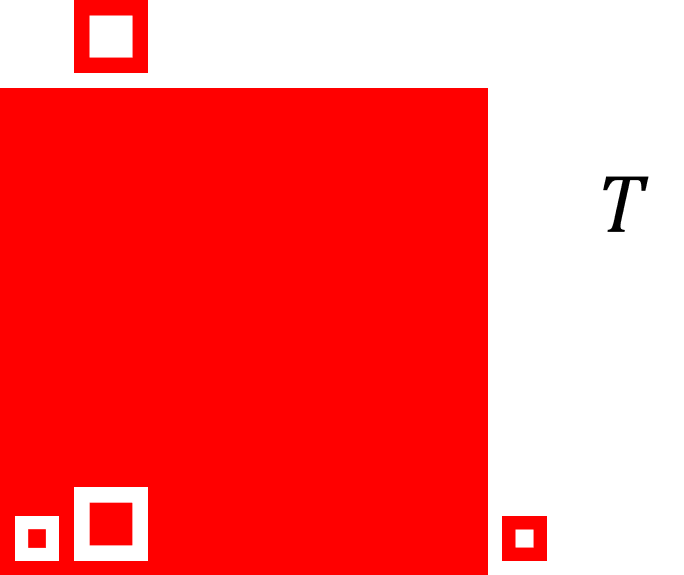}
		\caption{For a specific choice of vectors $v_i$'s, the tile $T$ from \eqref{eq:T} is illustrated here for $d=2$.}\label{Fig:T}
\end{figure}

We refer to the missing translates of $S_i$'s in $T$ as \emph{dents}, and to the disconnected translates of the $S_i$'s as \emph{bumps}.
We claim that $\Tile_0(T;\Z^d)= \{N\Z^d\}$. 
The claim is fairly evident from \Cref{Fig:T}. We provide some details below.

The properties of $T$ that we use, which are straightforward from \eqref{eq:T}, are the following: Firstly, $T$ is a fundamental domain for $N\Z^d$. Secondly, for every $1 \le i \le d$ the only $v \in T$ such that $v + S_i \cap T = \emptyset$ is $v_i$.  Thus,  $Ne_i + T$ is the unique translate of $T$ that can cover $Ne_i$ without intersecting the corresponding bump $Ne_i + v_i + S_i$. See \Cref{Fig:T}.
Since $T$ is a fundamental domain for $N\Z^d$, we have $N \Z^d \in \Tile_0(T;\Z^d)$.
Now suppose $C \in  \Tile_0(T;\Z^d)$, then for every $v \in C$ and $1\le i \le d$ we also have $v\pm Ne_i \in C$, because each bump of $v+T$ must fit into a corresponding dent and vice versa. So $C$ must be a union of cosets of $N\Z^d$.
But since $T$ is a fundamental domain for $N\Z^d$, it follows that $C$ must  be a coset of $N\Z^d$. Since $0 \in C$ we must have $C= N\Z^d$.

We now define the tiles $T_1,\ldots,T_s$.
For $1 \le j \le s$ the tile $T_j$ is defined by adding additional $r$ bumps and $r$ dents to the tile $T$. The bumps and dents are of the form $S_l$, for $d+r(j-1)< l\le d+jr$:
\begin{equation}\label{eq:Tj}
T_j := \left( T\setminus \biguplus_{l=1}^{r}\left(v_{d+r(j-1)+l} + S_{d+r(j-1)+l} \right) \right)
 \uplus \biguplus_{l=1}^r\left(N w_l + v_{d+r(j-1)+l} + S_{d+r(j-1)+l} \right).
\end{equation}

\begin{figure}[ht!]
		\includegraphics[scale=0.5]{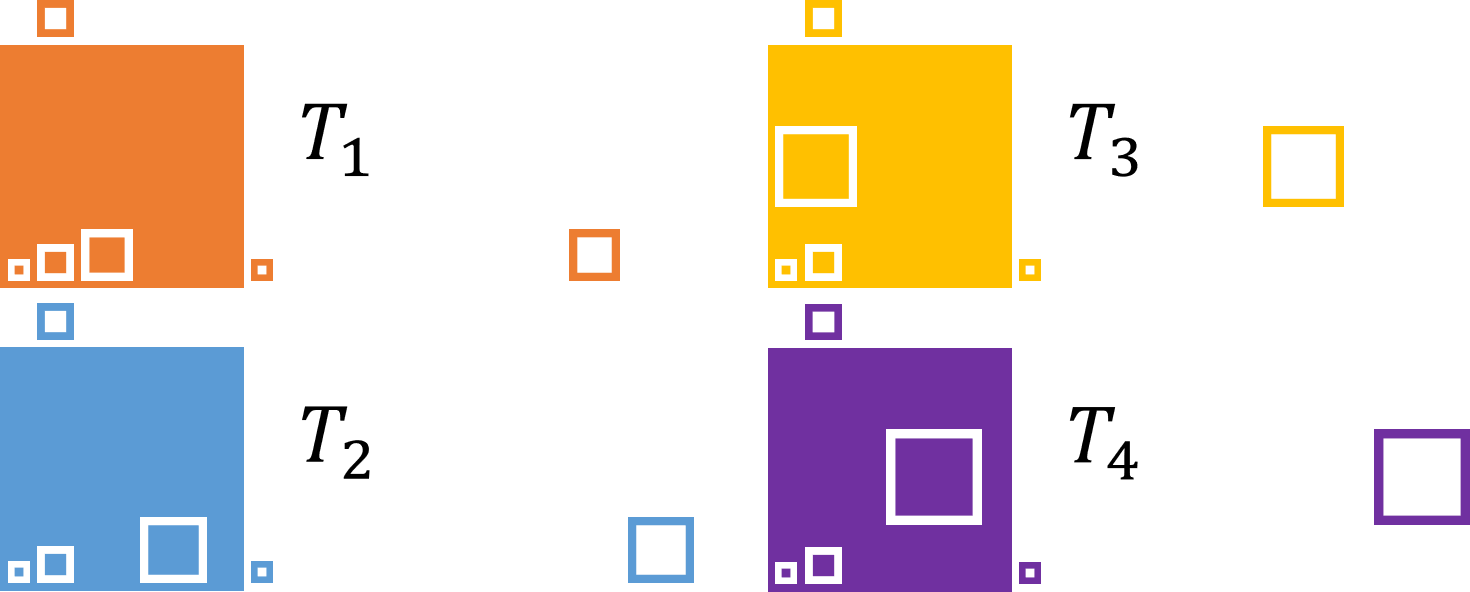}
		\caption{With the same parameters as in \Cref{Fig:T}, the tiles $T_j$ for $j\in\{1,2,3,4\}$, from \eqref{eq:Tj}, are illustrated above. Here $s=4$, $L$ is of rank $r=1$ and $w_1=(2,0)$.}\label{Fig:Tj}
	\end{figure}
	
Direct inspection of \eqref{eq:Tj} shows that for any $1\le j \le s$
we have
\begin{equation}\label{eq:T_properties}
T_j \oplus NL = T  \oplus NL
\end{equation}
From \eqref{eq:T_properties} and $T \oplus N\Z^d=\Z^d$ it follows that whenever $\{\tilde C_1,\ldots,\tilde C_s\}$ is a partition of $N\Z^d$ into $NL$-periodic sets, we have $(\tilde C_1,\ldots,\tilde C_s) \in \Tile_0(T_1,\ldots,T_s;\Z^d)$.

Now we show that for any $(\tilde C_1,\ldots,\tilde C_s) \in \Tile_0(T_1,\ldots,T_s;\Z^d)$ we have that  $\{\tilde C_1,\ldots,\tilde C_s\}$ is a partition of $N\Z^d$ into $NL$-periodic sets. 
As soon as we show that each $\tilde C_j$ is $NL$-periodic, it will follow from \eqref{eq:T_properties} that $\tilde C_j \oplus T_j = \tilde C_j \oplus T$ and so $\biguplus_{j=1}^s \tilde C_j \in \Tile_0(T;\Z^d)$, thus $\biguplus_{j=1}^s \tilde C_j = N\Z^d$. 
So it is left to explain why each $\tilde C_j$ is $NL$-periodic. The argument is essentially the same as the argument showing that each $C \in \Tile_0(T;\Z^d)$ is $N\Z^d$-periodic, and should be fairly clear from \Cref{Fig:Tj}.
In view of \eqref{eq:Tj}, the additional property of the collection $T_1,\ldots,T_s$ that is needed here is the following: For every $1 \le l \le r$ and every $1\le j,j' \le s$, if $v+T_{j'}$ covers $v_{d+r(j-1)+l}$ without intersecting the bump $Nw_l + v_{d+r(j-1)+l} + S_{d+r(j-1)+l}$ of $T_j$, then we must have $j=j'$ and $v=Nw_l$. See \Cref{Fig:Tj}.
\end{proof}

 \Cref{Fig:tildeF} illustrates the construction of the tiles $\tilde F_j\subset \Z^d$, as defined in \eqref{eq:tiling_reduction1}. 
\begin{figure}[ht!]
		\includegraphics[scale=0.4]{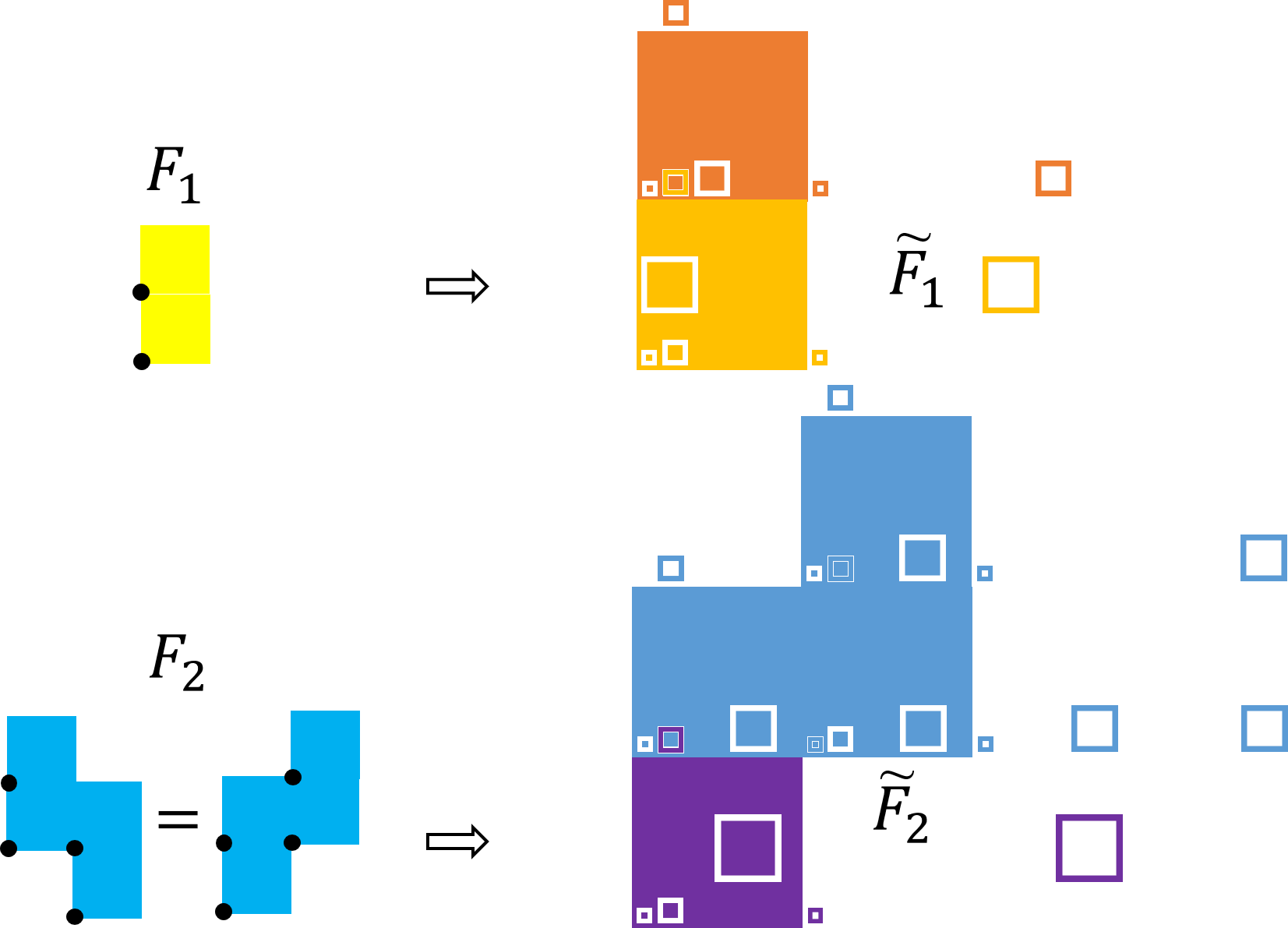}
		\caption{For $d=2$, $s=4$ and $\Gamma = (\Z/2\Z) \times \Z$, the figure illustrates a particular choice of tiles $F_1,F_2$ that can tile the ``vertical strip" $\Gamma$, and the corresponding construction of $\tilde F_1$ and $\tilde F_2$, using $T_1,T_2,T_3$ and $T_4$ from \Cref{Fig:Tj}. The $T_j$'s are constructed in this case using $L = \ker(\pi) =  \langle(2,0)\rangle$. Note the the two pictures of $F_2$ represent the same subset of $\Gamma$. }\label{Fig:tildeF}
	\end{figure}

For convenience we state and proof the following elementary lemma: 
\begin{lemma}\label{lem:direct_sum_pi}
Let $\pi:\Z^d \to \Gamma$ be a surjective group homomorphism, and let  $L = \ker(\pi) < \Z^d$.
\begin{enumerate}[label=(\alph*)]
    \item 
    If $A,B,C \subseteq \Z^d$ satisfy that $A \oplus B = C$ and $B$ is $L$-periodic then $\pi(A) \oplus \pi(B) = \pi(C)$.
    \item 
    If $C_1,\ldots,C_k \subseteq \Z^d$ are each $L$-periodic and $\biguplus_{j=1}^k C_j = C$ then $\biguplus_{j=1}^k \pi(C_j) = \pi(C)$.
\end{enumerate}
\end{lemma}

\begin{proof}
\begin{enumerate}[label=(\alph*)]
    \item 
    The equality  $\pi(A)+\pi(B) = \pi(C)$ follows directly because $\pi$ is a group homomorphism. We need to show that the sum is direct, namely that the representation $\pi(c)=\pi(a)+\pi(b)$ with $a \in A$, $b \in B$ is unique for every $c \in C$.
    Suppose $\pi(a_1)+\pi(b_1)=\pi(a_2)+\pi(b_2)$. Since every $c \in C$ has a unique representation $c=a+b$ with $a \in A$ and $b \in B$, we conclude that
    \[ 
    v:= (a_1+b_1) - (a_2 +b_2) \in L.
    \]
    Set $b_3 := b_2 +v$.  
    By the assumption that $B$ is $L$-periodic, it follows that $b_3 \in B$.
Then $a_2 +b_3 = a_1 + b_1$ are two representations of the same element as a sum of an element of $A$ and an element of $B$.
It follows that $a_1=a_2$ and $b_1=b_3$, so $b_1=b_2 +v$.
This means that $\pi(a_1)=\pi(a_2)$ and $\pi(b_1)=\pi(b_2)$.
We have thus proved that the representation $\pi(c)=\pi(a)+\pi(b)$ with $a \in A$, $b \in B$ is unique for every $c \in C$.
\item 
Using induction, it is enough to prove the case $k=2$. Clearly, $f(C_1\cup C_2) = f(C_1)\cup f(C_2)$ for every function $f$ and sets $C_1,C_2$ in its domain, so we only need to show that $\pi(C_1) \cap \pi(C_2) = \emptyset$ under the assumptions that $C_1 \cap C_2 = \emptyset$ and that each $C_i$ is $L$-periodic. Otherwise, there are $c_1\in C_1$ and $c_2\in C_2$ such that $\pi(c_1)=\pi(c_2)$. This means that $c_1 - c_2 \in L$. But $C_1$ is $L$-periodic and $c_1 \in C_1$ so $c_2 \in C_1$, contradicting $C_1 \cap C_2 =\emptyset$. 
\end{enumerate}
\end{proof}

We are now ready to prove our main result.
\begin{proof}[Proof of \Cref{thm:tiling_reduction}]
Given $d,k \in \N$ and a surjective homomorphism $\pi:\Z^d \to \Gamma$  let $L = \ker(\pi)$, let $D$ be a fundamental domain for $L$ that contains $0$, and let $T,T_1,\ldots,T_{2k} \subseteq \Z^d$ and $N \in \N$ satisfy the conclusion of \Cref{lem:rigid_tiles} with respect to $d,s=2k$ and $L$.

Given finite sets $F_1,\ldots,F_k \subset \Gamma$, let $\tilde F_1,\ldots, \tilde F_k \subset \Z^d$ be as in \eqref{eq:tiling_reduction1}. For the first inclusion, suppose $(A_1,\ldots,A_k) \in \Tile_0(F_1,\ldots,F_k;\Gamma)$, then by definition 
$\biguplus_{j=1}^k(F_j \oplus A_j) = \Gamma$. Applying $\pi^{-1}$, we obtain that
\begin{equation}\label{eq:F1,..,F_k_tile_Z^d}
\biguplus_{j=1}^k\pi^{-1}(F_j \oplus A_j) = \Z^d.
\end{equation}
For $1\le j \le k$ we define
\begin{equation}\label{eq:Cj}
\tilde C_j := \left( (N\pi^{-1}(F_j\setminus \{0\}) \cap D)\oplus N\pi^{-1}(A_j)\right) \mbox{ and }
\tilde C_{j+k} :=  N\pi^{-1}(A_j).
\end{equation}
By the identity 
%$\pi^{-1}(F_j \oplus A_j)= (\pi^{-1}(F_j) \cap D)\oplus \pi^{-1}(A_j)$, 
\[ 
\pi^{-1}(F_j \oplus A_j)= (\pi^{-1}(F_j) \cap D)\oplus \pi^{-1}(A_j)= (\pi^{-1}(F_j \setminus \{0\}) \cap D) \oplus \pi^{-1}(A_j)\uplus \pi^{-1}(A_j),
\]
and in view of \eqref{eq:Cj}, we see that for every $1\le j\le k$ we have 
$\pi^{-1}(F_j \oplus A_j) = \frac1N\left[\tilde C_j \uplus \tilde C_{j+k}\right]$. 
Thus multiplying the identity %$\biguplus_{j=1}^k\pi^{-1}(F_j \oplus A_j) = \Z^d$ 
in \eqref{eq:F1,..,F_k_tile_Z^d} by $N$ yields $\biguplus_{j=1}^{2k}\tilde C_j = N\Z^d$.

Note that $N\pi^{-1}(A_j)$ is $NL$-periodic, thus $\{\tilde C_1,\ldots,\tilde C_{2k}\}$ is a partition of $N\Z^d$ into $NL$-periodic sets. By  \Cref{lem:rigid_tiles} it follows that
\[
\left(\tilde C_1,\ldots, \tilde C_{2k} \right) \in \Tile_0(T_1,\ldots,T_{2k};\Z^d).
\]
Using \eqref{eq:Cj}, the expression 
$\biguplus_{j=1}^{2k} T_j \oplus \tilde C_j = \Z^d$ 
can be rearranged to 
\[
\biguplus_{j=1}^k \left[ \left((N \pi^{-1}(F_j\setminus \{0\}) \cap D) \oplus T_j \right) \uplus T_{j+k} \right]\oplus N\pi^{-1}(A_j) = \Z^d.
\]
But in view of the definition of the $\tilde F_j$ in \eqref{eq:tiling_reduction1}, this means that 
 $(N\pi^{-1}(A_1),\ldots,N\pi^{-1}(A_k)) \in \Tile_0(\tilde F_1,\ldots,\tilde F_k; \Z^d)$. We conclude that
\[
 N \Big( (\pi^{\otimes k})^{-1} \Tile_0\left(F_1,\ldots,F_k; \Gamma\right) \Big) \subseteq \Tile_0(\tilde F_1,\ldots,\tilde F_k; \Z^d).
\]
To see the other inclusion, let $(\tilde A_1,\ldots,\tilde A_k) \in \Tile_0(\tilde F_1,\ldots,\tilde F_k; \Z^d)$. Then the  following statements hold:
\begin{enumerate}
    \item 
    $0 \in \biguplus_{j=1}^k \tilde A_j$.
    \item 
    $\biguplus_{j=1}^k \tilde F_j \oplus \tilde A_j = \Z^d$.
\end{enumerate}
The second equation can be rewritten as follows:

\[
\biguplus_{j=1}^k \left[ \left((N \pi^{-1}(F_j\setminus \{0\}) \cap D) \oplus T_j \right) \uplus T_{j+k} \right]\oplus \tilde A_j = \Z^d.
\]
This can be rearranged as:
\begin{equation}\label{eq:A_tilde_tiling}
\biguplus_{j=1}^k \left(N(\pi^{-1}(F_j\setminus \{0\}) \cap D) \oplus T_j \oplus \tilde A_j\right) \uplus \biguplus_{j=1}^k  (T_{j+k} \oplus \tilde A_j) = \Z^d.
\end{equation}
Thus $(\tilde A_1,\ldots,\tilde A_k) \in \Tile_0(\tilde F_1,\ldots,\tilde F_k; \Z^d)$ if and only if 
\[\left( \tilde C_1, \ldots, \tilde C_{2k}\right) \in \Tile_0(T_1,\ldots,T_{2k}; \Z^d),\]
where
\[
\tilde C_j := N(\pi^{-1}(F_j\setminus\{0\}) \cap D)  \oplus \tilde A_j \mbox{ and  }\tilde C_{j+k} := \tilde A_j ~\text{ for } 1\le j \le k.
\]
By \Cref{lem:rigid_tiles}, this happens if and only if $\{\tilde C_1,\ldots,\tilde C_{2k}\}$ is a partition of $N\Z^d$ into $NL$-periodic sets.
Since $\tilde C_j \subseteq N\Z^d$ it follows that $\tilde A_j \subseteq N\Z^d$. Now because $\tilde C_{j+k}$ is $NL$-periodic, it follows that each $\tilde A_j$ is $NL$-periodic. This means that for every $1\le j \le k$ there exists $A_j \subseteq \Gamma$ such that
\begin{equation}\label{eq:tildeAj_formula}
\tilde A_j= N \pi^{-1}(A_j).
\end{equation}
It remains to show that $(A_1,\ldots,A_k) \in \Tile_0(F_1,\ldots,F_k;\Gamma)$. Because $0 \in \biguplus_{j=1}^k \tilde A_j$, it follows that $0 \in \biguplus_{j=1}^k A_j$. Using the property $T_l \oplus NL= T \oplus NL$ for every $1\le l\le 2k$, and the fact that each $\tilde A_j$ is $NL$-periodic, we may replace every $T_l$ in \eqref{eq:A_tilde_tiling} by $T$ to  obtain:
\[
\biguplus_{j=1}^k \left(N(\pi^{-1}(F_j\setminus \{0\}) \cap D) \oplus T \oplus \tilde A_j\right) \uplus \biguplus_{j=1}^k  (T \oplus \tilde A_j) = \Z^d. 
\]
Using the fact that $0\in D$, the left-hand side  can be rearranged to obtain the following:
\[
\left[\biguplus_{j=1}^k \left(N(\pi^{-1}(F_j) \cap D)  \oplus \tilde A_j \right) \right] \oplus T = \Z^d.
\]
Recall that by \Cref{lem:rigid_tiles} we have $\Tile_0(T;\Z^d) = \{N\Z^d\}$, thus 
\[
\biguplus_{j=1}^k \left(N(\pi^{-1}(F_j) \cap D)  \oplus \tilde A_j \right)  = N\Z^d.
\] 
Plugging in \eqref{eq:tildeAj_formula} we conclude that  $\biguplus_{j=1}^k \left(N(\pi^{-1}(F_j) \cap D)  \oplus N \pi^{-1}(A_j) \right)  = N\Z^d$. Since $v \mapsto N v$ induces a group isomorphism between $\Z^d$ and $N\Z^d$ we get $\biguplus_{j=1}^k \left((\pi^{-1}(F_j) \cap D)  \oplus  \pi^{-1}(A_j) \right)  = \Z^d$. By \Cref{lem:direct_sum_pi}, using the fact that $\pi^{-1}(A_j)$ is $L$-periodic, it follows that 
$\biguplus_{j=1}^k F_j \oplus A_j = \Gamma$.
We have thus shown that
\[
\Tile_0(\tilde F_1,\ldots,\tilde F_k; \Z^d) \subseteq 
N \left( (\pi^{\otimes k})^{-1} \Tile_0\left(F_1,\ldots,F_k;\Gamma\right) \right).
\] 
\end{proof} 

\begin{figure}[ht!]
		\includegraphics[scale=0.5]{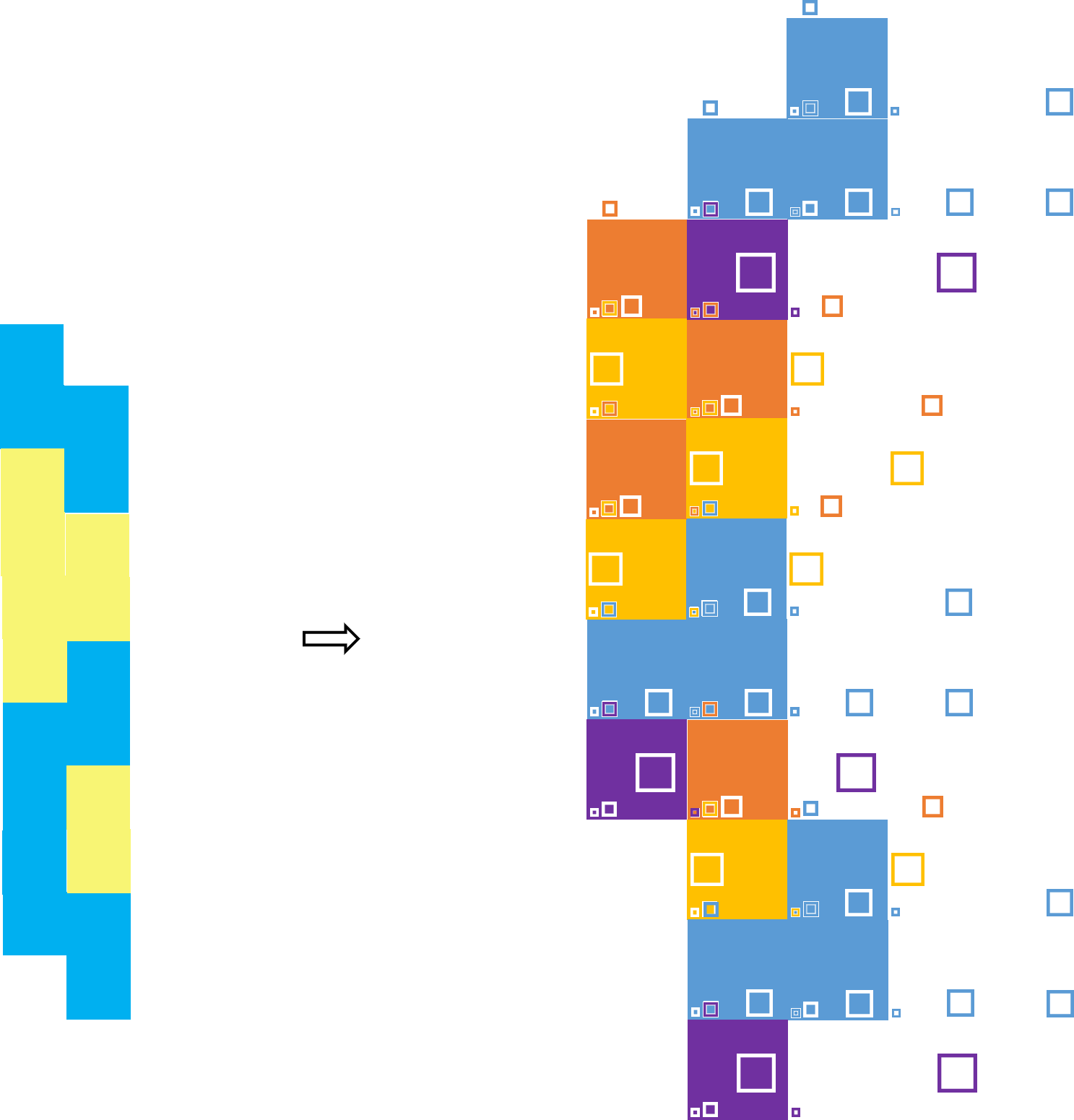}
		\caption{A patch of $\Gamma$-tiling using $F_1$ and $F_2$ and a patch of the corresponding $\Z^2$ -tiling using $\tilde F_1$ and $F_2$.
		%The additional role of the tiles $T_3$ and $T_4$ (the $T_{j+k}$ in the proof of \Cref{thm:tiling_reduction}) is to indicate a specific point - ``the center", in each $\tilde F_j$. Observe that if $\tilde F_1$ here was made of two copies of $T_1$, a tiling of $\Gamma$ by $F_1$ would not force an $L$-periodic tiling of $\Z^2$ by $\tilde F_1$.% thus the necessity of the tiles $T_{j+k}$ follow.
		}\label{Fig:Strip}
	\end{figure}

\section{Corollaries regarding periodicity, topological conjugacy, algorithmic decidability and more}

\begin{proof}[Proof of \Cref{cor:periodicity_reduction}] 
Suppose $(F_1,\ldots,F_k)$ is an aperiodic $k$-tuple of finite sets in $\Gamma$, then $\Tile(F_1,\ldots,F_k;\Gamma) \ne \emptyset$. By \eqref{eq:tiling_reduction1}, we have that $\Tile(\tilde F_1,\ldots,\tilde F_k;\Z^d) \ne \emptyset$. If we assume that $(\tilde F_1,\ldots,\tilde F_k)$ is not aperiodic, then by \eqref{eq:tiling_reduction1} there exists $(A_1,\ldots,A_k) \in \Tile(F_1,\ldots,F_k;\Gamma)$ such that $(N\pi^{-1}(A_1),\ldots,N\pi^{-1}(A_k))$ is a periodic $k$ tuple in $\Z^d$. But this implies that $(A_1,\ldots,A_k)$ is periodic. Thus, the existence of an aperiodic $k$-tuple in $\Z^d$ implies the existence of an aperiodic $k$-tuple in $\Gamma$.
\end{proof}

\begin{proof}[Proof of \Cref{cor:algorithmic_decidability_reduction}]
Assume that for a particular $k \in \mathbb{N}$ a tiling by $k$ tiles in $\Z^d$ is algorithmically decidable. By definition, this means that there exists a Turing machine $\tilde M$ that takes as input a $k$-tuple $\tilde F_1,\ldots,\tilde F_k$ of finite subsets in $\Z^d$ and accepts the input if and only if $\Tile(\tilde F_1,\ldots,\tilde F_k; \Z^d) \ne \emptyset$. The Turing machine $\tilde M$ always halts in finite time. Then we can construct a Turing machine $M$ that takes as input a $k$-tuple $F_1,\ldots,F_k$ of finite subsets in $\Gamma$, produces the finite subsets $\tilde F_1,\ldots,\tilde F_k$ of $\Z^d$ described in the statement of \Cref{thm:tiling_reduction}, then runs $\tilde M(\tilde F_1,\ldots,\tilde F_k)$. Clearly $M$ halts in finite time on input $F_1,\ldots,F_k$ if and only if $\tilde M$ halts with input $\tilde F_1,\ldots,\tilde F_k$. The Turing machine $M$ accepts $F_1,\ldots,F_k$ if and only if $\tilde M$ accepts $\tilde F_1,\ldots,\tilde F_k$, which by assumption happens if and only if $\Tile(\tilde F_1,\ldots,\tilde F_k; \Z^d) \ne \emptyset$. By \Cref{thm:tiling_reduction} this happens if and only if $\Tile(F_1,\ldots,F_k; \Z^d) \ne \emptyset$.
 \end{proof}
 
 \begin{proof}[Proof of \Cref{cor:logical_decidability}]
 The concatenation of the proof of $\Tile(\tilde F_1,\ldots,\tilde F_k; \Z^d) \ne \emptyset$ in $\mathcal{T}$ with the proof of \eqref{eq:tiling_reduction1} in $\mathcal{T}$ yields a proof of $\Tile(F_1,\ldots,F_k;\Gamma) \ne \emptyset$ in $\mathcal{T}$.
 \end{proof}
 
 To conclude this note, we explain why \Cref{cor:periodicity_reduction}  is a consequence of a stronger correspondence between translational tilings in $\Gamma$ and in $\Z^d$,  expressed in terms of topological dynamics.
 
 A \emph{$\Z^d$-topological dynamical system} is a pair $(X,T)$ where $X$ is a compact topological space and $T \in \mathit{Hom}(\Z^d,\mathit{Homeo}(X))$, the group of homomorphisms from $\Z^d$ to the group $\mathit{Homeo}(X)$ of self-homeomorphisms of $X$. We say that $\Z^d$-topological dynamical systems $(X,T)$ and $(Y,S)$ are \emph{topologically conjugate} or \emph{isomorphic} if there exists a homeomorphism $\Phi:X \to Y$ such that $S_v= \Phi \circ T_v \circ \Phi^{-1}$ for every $v \in \Z^d$.

 The space $(2^\Gamma)^k$ of $k$-tuples of finite subsets of $\Gamma$, equipped with the product topology, is a totally disconnected compact metrizable space. 
 Note that that wherever $\pi:\Z^d \to \Gamma$ is a surjective group homomorphism, the group $\Z^d$ acts on $(2^\Gamma)^k$ by homeomorphism via $\sigma^{(\pi)} \in \mathit{Hom}(\Z^d,\mathit{Homeo}(2^\Gamma)^k)$, where for every $v \in \Z^d$, $\sigma^{(\pi)}_v \in \mathit{Homeo}((2^\Gamma)^k)$ is given by 
 \[
\sigma^{(\pi)}_v(A_1,\ldots,A_k)= (A_1+\pi(v),\ldots,A_k+\pi(v)) \mbox{ for } (A_1,\ldots,A_k) \in (2^\Gamma)^k.
 \]

 For any $k$-tuple $(F_1,\ldots,F_k)$ of finite sets the space $\Tile(F_1,\ldots,F_k;\Gamma)$ is a closed, $\sigma^{(\pi)}$-invariant  subset of $(2^\Gamma)^k$. So $(\Tile(F_1,\ldots,F_k;\Gamma),\sigma^{(\pi)})$ is a $\Z^d$-topological dynamical system. 
 Let $N \in \N$ and $(\tilde F_1,\ldots,\tilde F_k)$ be a $k$-tuple of finite subsets of $\Z^d$ that satisfies \eqref{eq:tiling_reduction1}. Let 
 
 $\sigma, \sigma^{(N)} \in \mathit{Hom}(\Z^d,\mathit{Homeo}((2^{\Z^d})^k))$ be given by 
 \[
 \sigma_v(A_1,\ldots,A_k) = (A_1+v,\ldots,A_k +v) \mbox{ for } (A_1,\ldots,A_k) \in (2^{\Z^d})^k.
 \]
 and
 \[
\sigma^{(N)}_v(A_1,\ldots,A_k)= (A_1+Nv,\ldots,A_k+N v) \mbox{ for } (A_1,\ldots,A_k) \in (2^{\Z^d})^k.
 \]
 
 Let $D_N := \{0,\ldots,N-1\}^d \subset \Z^d$. 
 For each $t \in D_N$ let
 
 \[ X_t = \bigcup_{v \in N\Z^d +t}\sigma_v \left( \Tile_0(\tilde F_1,\ldots,\tilde F_k;\Z^d)\right).\]
Then each $X_t$ is a closed, $\sigma^{(N)}$-invariant subset of $\Tile(\tilde F_1,\ldots,\tilde F_k;\Z^d)$, and
$\Tile(\tilde F_1,\ldots,\tilde F_k; \Z^d)= \biguplus_{t \in D_N} X_t$. 
For each $t \in D_N$, the $\Z^d$-topological dynamical system  $(X_t,\sigma^{(N)})$ is topologically conjugate to $(\Tile(F_1,\ldots,F_k;\Gamma),\sigma^{(\pi)})$ via the map
\[ (A_1,\ldots, A_k) \mapsto \left( \pi\left(\frac{1}{N}(A_1 -t)\right),\ldots,\pi\left(\frac{1}{N}(A_k-t)\right)\right).\]

We conclude:

\begin{cor}\label{cor:dynamics_reduction}
For every $k$-tuple of finite subsets of $(F_1,\ldots,F_k)$ of $\Gamma$ there exists a $k$-tuple  $(\tilde F_1,\ldots,\tilde F_k)$ of finite subsets of $\Z^d$ and $N \in \N$  such that 
%$\Tile(\tilde F_1,\ldots,\tilde F_k;\Gamma)$ such that 
the $\Z^d$-topological dynamical system $(\Tile(\tilde F_1,\ldots,\tilde F_k; \Z^d),\sigma^{(N)})$ is topologically conjugate to the $\Z^d$-topological dynamical system $(\Tile(F_1,\ldots,F_k; \Gamma)\times \{1,\ldots,N^d\} ,\sigma^{(\pi)}\times \mathit{Id})$.
\end{cor}
\bibliographystyle{alpha}
\bibliography{lib}
\end{document}